\documentclass[12pt, reqno]{amsart}

\author[M.~Balcerzak]{Marek Balcerzak}
\address{Institute of Mathematics, \L{}\'{o}d\'{z} University of Technology, ul. W\'{o}lcza\'{n}ska 215, 93-005 \L{}\'{o}d\'{z}, Poland}
\email{marek.balcerzak@p.lodz.pl}
\author[Sz.~G{\l}\k{a}b]{Szymon G{\l}\k{a}b}
\address{Institute of Mathematics, \L{}\'{o}d\'{z} University of Technology, ul. W\'{o}lcza\'{n}ska 215, 93-005 \L{}\'{o}d\'{z}, Poland}
\email{szymon.glab@p.lodz.pl}
\author[P.~Leonetti]{Paolo Leonetti}
\address{Department of Decision Sciences, Universit\`a Luigi Bocconi, via Roentgen 1, 20136 Milan, Italy}
\email{leonetti.paolo@gmail.com}


\title{Another characterization of meager ideals}

\usepackage[T1]{fontenc}
\usepackage{amsmath}
\usepackage{amssymb}
\usepackage{amsthm}
\usepackage[left=3.5cm, right=3.5cm, paperheight=11.8in]{geometry}
\usepackage{hyperref}
\usepackage{fancyhdr}
\usepackage{enumitem}
\usepackage{comment}
\usepackage{nicefrac}
\usepackage{mathrsfs}
\usepackage{graphicx}
\usepackage[utf8]{inputenc}

\AtBeginDocument{%
   \def\MR#1{}
}

\newtheorem{thm}{Theorem}[section]
\newtheorem{cor}[thm]{Corollary}
\newtheorem{lem}[thm]{Lemma}

\theoremstyle{definition} 
\newtheorem{defi}[thm]{Definition}
\let\olddefi\defi
\renewcommand{\defi}{\olddefi\normalfont}
\newtheorem{example}[thm]{Example}
\let\oldexample\example
\renewcommand{\example}{\oldexample\normalfont}
\newtheorem{rmk}[thm]{Remark}
\let\oldrmk\rmk
\renewcommand{\rmk}{\oldrmk\normalfont}

\newcommand{\clusterfin}{\mathrm{L}_x}

\hypersetup{
    pdftitle={},
    pdfauthor={},
    pdfmenubar=false,
    pdffitwindow=true,
    pdfstartview=FitH,
    colorlinks=true,
    linkcolor=blue,
    citecolor=green,
    urlcolor=cyan
}

\providecommand{\MR}[1]{}

\providecommand{\MR}{\relax\ifhmode\unskip\space\fi MR }

\providecommand{\href}[2]{#2}

\thanks{P.L. is grateful to PRIN 2017 (grant 2017CY2NCA)
for financial support.}
                                 
\begin{document}

\maketitle
\thispagestyle{empty}

\begin{abstract}
\noindent We show that an ideal $\mathcal{I}$ on the positive integers is meager if and only if there exists a bounded nonconvergent real sequence $x$ such that the set of subsequences [resp. permutations] of $x$ which preserve the set of $\mathcal{I}$-limit points is comeager and, in addition, every accumulation point of $x$ is also an $\mathcal{I}$-limit point (that is, a limit of a subsequence $(x_{n_k})$ such that $\{n_1,n_2,\ldots,\} \notin \mathcal{I}$). The analogous characterization holds also for $\mathcal{I}$-cluster points.  
\end{abstract}

\section{Introduction}\label{sec:intro}

By a known result due to Buck \cite{MR9997}, almost every subsequence, in the sense of measure, of a given real sequence $x$ has the same set of ordinary limit points of the original sequence $x$. 
Extensions and other measure-related results may be found in \cite{MR3568092, 
MR3879311, 
Leo17b, 
MR1260176, 
MR4089884, 
MR1924673}. 
The aim of this note is to prove its topological [non]analogue in the context of ideal convergence, following the line of research in \cite{MR4165727, MR3968131, MR3950736, MR4183385}. This will allow us to obtain a characterization of meager ideals in Theorem \ref{cor:characterizationmeager}. 

We recall briefly the main definitions given in \cite{MR4165727}. 
Let $\mathcal{I}$ be an ideal on the positive integers $\mathbf{N}$, that is, a proper subset of $\mathcal{P}(\mathbf{N})$ closed under taking finite unions and subsets and containing the family $\mathrm{Fin}$ of finite sets. Ideals will be regarded as subsets of the Cantor space $\{0,1\}^{\mathbf{N}}$ endowed with the product topology. 
Let also $x=(x_n)$ be a sequence taking values in a topological space $X$, which will be always assumed Hausdorff. 
Then, denote by $\Gamma_x(\mathcal{I})$ the set of its $\mathcal{I}$-cluster points, that is, the set of all $\eta \in X$ such that $\{n \in \mathbf{N}: x_n \in U\} \notin \mathcal{I}$ for all neighborhoods $U$ of $\eta$. Lastly, let $\mathrm{L}_x:=\Gamma_x(\mathrm{Fin})$ be the set of ordinary accumulation points of $x$. 

Following the same notations in \cite{MR4165727}, define
$$
\Sigma:=\{\sigma \in \mathbf{N}^{\mathbf{N}}: \sigma \text{ is strictly increasing\hspace{.5mm}}\}.
$$ 
For each $\sigma \in \Sigma$, we denote by $\sigma(x)$ the subsequence $(x_{\sigma(n)})$. We identify each subsequence of $(x_{k_n})$ of $x$ with the function $\sigma \in \Sigma$ defined by $\sigma(n)=k_n$ for all $n \in \mathbf{N}$. 
Similarly, define 
$$
\Pi:=\{\pi \in  \mathbf{N}^{\mathbf{N}}: \pi \text{ is a bijection\hspace{.5mm}}\}
$$  
and write $\pi(x)$ for the rearranged sequence $(x_{\pi(n)})$. Endow both $\Sigma$ and $\Pi$ with their relative topology and note, since they are $G_\delta$-subsets of $\mathbf{N}^{\mathbf{N}}$, they are Polish spaces by Alexandrov's theorem. 
In particular, they are not meager in themselves. 
Finally, denote by
$$
\Sigma_x(\mathcal{I}):=\left\{\sigma \in \Sigma: \Gamma_{\sigma(x)}(\mathcal{I})=\Gamma_x(\mathcal{I})\right\}
$$ 
the set of subsequences of $x$ which preserve the $\mathcal{I}$-cluster points of $x$, and by
$$
\Pi_x(\mathcal{I}):=\left\{\pi \in \Pi: \Gamma_{\pi(x)}(\mathcal{I})=\Gamma_x(\mathcal{I})\right\}
$$ 
its permutation analogue. 

The following result has been shown in \cite[Theorem 2.2]{MR4165727}:
\begin{thm}\label{thm:theoremcluster}
Let $x$ be a sequence taking values in a first countable space $X$ such that all closed sets are separable and let $\mathcal{I}$ be a meager ideal on $\mathbf{N}$. 

Then the following are equivalent\textup{:} 
\begin{enumerate}[label={\rm (\textsc{c}\arabic{*})}]
\item \label{item:c1} $\Sigma_x(\mathcal{I})$ is comeager\textup{;}
\item \label{item:c2} $\Sigma_x(\mathcal{I})$ is not meager\textup{;}
\item \label{item:c3} $\Pi_x(\mathcal{I})$ is comeager\textup{;}
\item \label{item:c4} $\Pi_x(\mathcal{I})$ is not meager\textup{;}
\item \label{item:c5} $\Gamma_x(\mathcal{I})=\clusterfin$\textup{.}
\end{enumerate}
\end{thm}

Given a sequence $x=(x_n)$ taking values in a topological space $X$, we denote by $\Lambda_x(\mathcal{I})$ the set of $\mathcal{I}$-limit points of $x$, that is, the set of all $\eta \in X$ such that $\lim \sigma(x)=\eta$ for some $\sigma \in \Sigma$ such that $\sigma[\mathbf{N}] \notin \mathcal{I}$. 
It is well known that $\Lambda_x(\mathcal{I}) \subseteq \Gamma_x(\mathcal{I}) \subseteq \mathrm{L}_x$. 
Further relationships between $\mathcal{I}$-cluster points and $\mathcal{I}$-limit points have been studied in \cite{MR3883171}. 
Similarly, we denote by 
$$
\tilde{\Sigma}_x(\mathcal{I}):=\left\{\sigma \in \Sigma: \Lambda_{\sigma(x)}(\mathcal{I})=\Lambda_x(\mathcal{I})\right\}
$$
and
$$
\tilde{\Pi}_x(\mathcal{I}):=\left\{\pi \in \Pi: \Lambda_{\pi(x)}(\mathcal{I})=\Lambda_x(\mathcal{I})\right\}
$$
the analogues of $\Sigma_x(\mathcal{I})$ and $\Pi_x(\mathcal{I})$, respectively, for $\mathcal{I}$-limit points. 

Our first main result is the exact analogue of Theorem \ref{thm:theoremcluster} for $\mathcal{I}$-limit points.
\begin{thm}\label{thm:theoremlimit}
Let $x$ be a sequence taking values in a first countable space $X$ such that all closed sets are separable and let $\mathcal{I}$ be a meager ideal on $\mathbf{N}$. 
 
Then the following are equivalent\textup{:} 
\begin{enumerate}[label={\rm (\textsc{l}\arabic{*})}]
\item \label{item:l1} $\tilde{\Sigma}_x(\mathcal{I})$ is comeager\textup{;}
\item \label{item:l2} $\tilde{\Sigma}_x(\mathcal{I})$ is not meager\textup{;}
\item \label{item:l3} $\tilde{\Pi}_x(\mathcal{I})$ is comeager\textup{;}
\item \label{item:l4} $\tilde{\Pi}_x(\mathcal{I})$ is not meager\textup{;}
\item \label{item:l5} $\Lambda_x(\mathcal{I})=\clusterfin$\textup{.}
\end{enumerate}
\end{thm}
We remark that Theorem \ref{thm:theoremlimit} provides an affirmative answer to the open question stated at the end of Section 2 in \cite{MR4165727}. In addition, several special cases of Theorem \ref{thm:theoremcluster} have been already obtained in the literature:
\begin{enumerate}[label={\rm (\roman{*})}]
\item \cite[Theorem 2.3]{MR3950736} for the case where $X=\mathbf{R}$, $\mathcal{I}=\mathcal{Z}:=\{A\subseteq \mathbf{N}: \lim_n |A\cap [1,n]|/n=0\}$, and the equivalences \ref{item:l1} $\Longleftrightarrow$ \ref{item:l2} $\Longleftrightarrow$ \ref{item:l5};
\item \cite[Theorem 2.3]{MR3968131} for the case where $\mathcal{I}$ is a generalized density ideal and the equivalences \ref{item:l1} $\Longleftrightarrow$ \ref{item:l2} $\Longleftrightarrow$ \ref{item:l5};
\item \cite[Theorem 2.9]{MR4165727} for the case where $\mathcal{I}$ is an analytic P-ideal;
\item \cite[Theorem 1]{MR4183385} for the case where $X=\mathbf{R}$ with the equivalence \ref{item:l1} $\Longleftrightarrow$ \ref{item:l5}.
\end{enumerate} 

At this point, as it has been shown in \cite[Example 2.6]{MR4165727}, it is worth noting that, if $\mathcal{I}$ is maximal (that is, the complement of a free ultrafilter), 
then there exists a bounded real sequence $x$ which satisfies \ref{item:c2} 
but not \ref{item:c5}, cf. Remark \ref{rmk:nostrongercorollary} below. Hence the statement of Theorem \ref{thm:theoremcluster} (and similarly for Theorem \ref{thm:theoremlimit}) certainly does not apply to all ideals. 

We are going to show below that this is not a coincidence, roughly meaning that the hypothesis of meagerness of the ideal $\mathcal{I}$ is essential for all the above equivalences. In other words, this provides a characterization of meager ideals, which is our second main result. 

Before we present this characterization, we recall a similar result on series of real numbers. To this aim, we need some additional notations: for each real sequence $x$, let $Sx$ be the sequence of its partial sums, that is, $Sx=(S_nx)$ with $S_nx:=\sum_{i\le n}x_i$ for all $n \in \mathbf{N}$. The sequence $x$ is said to be $\mathcal{I}$-bounded if $\{n \in \mathbf{N}: |x_n|>k\} \in \mathcal{I}$ for some $k \in \mathbf{N}$. 
The vector space of $\mathcal{I}$-bounded sequences, denoted by $\ell_\infty(\mathcal{I})$, has been studied, e.g., in \cite{MR2735533, MR2135846}. Then, define 
$$
\Sigma_{S,x}(\mathcal{I}):=\{\sigma \in \Sigma: S\sigma(x) \in \ell_\infty(\mathcal{I})\}.
$$
Accordingly, the following result has been shown in \cite[Theorem 4.1]{MR3712964}.
\begin{thm}\label{thm:poplawski}
Let $x$ be a real sequence such that the series $\sum_nx_n$ is not unconditionally convergent and let $\mathcal{I}$ be a meager ideal on $\mathbf{N}$. Then\textup{:}
\begin{enumerate}[label={\rm (\textsc{s}\arabic{*})}]
\item \label{item:s1} $\Sigma_{S,x}(\mathcal{I})$ is meager\textup{.}
\end{enumerate}
\end{thm}
A related result on series taking values in Banach spaces has been shown in \cite[Corollary 3.3]{MR3989357}, cf. also \cite[Theorem 3.4 and Example 2]{MR3712964}. It is worth to take in mind that, if $\liminf_n x_n>0$, then the series $\sum_nx_n$ is not unconditionally convergent and, for each $\sigma \in \Sigma$ and \emph{all} ideals $\mathcal{I}$, the sequence of partial sums $S\sigma(x)$ has limit infinity, hence $\Sigma_{S,x}(\mathcal{I})=\emptyset$; this example proves that the condition \ref{item:s1} alone cannot provide a characterization of meager ideals.

For the sake of exposition, let $\mathscr{A}$ be the set of all sequences taking values in some first countable space $X$ such that all closed sets are separable, and $\mathscr{B}$ be its subset of nonconvergent sequences with at least one (ordinary) accumulation point. Lastly, 
let 
$\mathscr{C}$ be the set of real sequences with dense image.

\begin{thm}\label{cor:characterizationmeager}
Let $\mathcal{I}$ be an ideal on $\mathbf{N}$. Then the following are equivalent\textup{:} 
\begin{enumerate}[label={\rm (\textsc{m}\arabic{*})}]
\item \label{item:m1} For all sequences $x \in \mathscr{A}$ we have  \ref{item:c1} $\Longleftrightarrow$ \ref{item:c5}\textup{;}
\item \label{item:m3} For all sequences $x \in \mathscr{A}$ we have  \ref{item:c3} $\Longleftrightarrow$ \ref{item:c5}\textup{;}
\item \label{item:m1L} For all sequences $x \in \mathscr{A}$ we have  \ref{item:l1} $\Longleftrightarrow$ \ref{item:l5}\textup{;}
\item \label{item:m3L} For all sequences $x \in \mathscr{A}$ we have  \ref{item:l3} $\Longleftrightarrow$ \ref{item:l5}\textup{;}
\item \label{item:m2} There exists a sequence $x \in \mathscr{B}$ such that both \ref{item:c1} and \ref{item:c5} hold\textup{;}
\item \label{item:m4} There exists a sequence $x \in \mathscr{B}$ such that both \ref{item:c3} and \ref{item:c5} hold\textup{;}
\item \label{item:m2L} There exists a sequence $x \in \mathscr{B}$ such that both \ref{item:l1} and \ref{item:l5} hold\textup{;}
\item \label{item:m4L} There exists a sequence $x \in \mathscr{B}$ such that both \ref{item:l3} and \ref{item:l5} hold\textup{;}
\item \label{item:mSeries} There exists a sequence $x \in \mathscr{C}$ such that \ref{item:s1} holds\textup{;}
\item \label{item:m5} $\mathcal{I}$ is meager\textup{.}
\end{enumerate}
\end{thm}

Considering that real bounded nonconvergent sequences belong to $\mathscr{B}$ and the equivalence \ref{item:m2L} $\Longleftrightarrow$ \ref{item:m5} in Theorem \ref{cor:characterizationmeager}, we obtain the following corollary (which is stated in the abstract):
\begin{cor}\label{cor:boundednonconvergent}
An ideal $\mathcal{I}$ on $\mathbf{N}$ is meager if and only if there exists a real bounded nonconvergent sequence $x$ such that $\Lambda_x(\mathcal{I})=\mathrm{L}_x$ and $\tilde{\Sigma}_x(\mathcal{I})$ is comeager.
\end{cor}

Note that the definition of $\mathscr{B}$ imposes that each of its elements has at least one accumulation point. This constraint cannot be removed. Indeed, if $x \in \mathscr{A}$ verifies $\mathrm{L}_x=\emptyset$, then $\Gamma_{\sigma(x)}(\mathcal{I})\subseteq \mathrm{L}_{\sigma(x)}\subseteq \mathrm{L}_x=\emptyset$ for all $\sigma \in \Sigma$. Hence $\Sigma_x(\mathcal{I})=\Sigma$, independently of the choice of the ideal $\mathcal{I}$. This would imply that both \ref{item:c1} and \ref{item:c5} hold true, also for nonmeager ideals, and it would provide a counterexample to the equivalence \ref{item:m2} $\Longleftrightarrow$ \ref{item:m5} in Theorem \ref{cor:characterizationmeager}. 

\begin{rmk}\label{rmk:nostrongercorollary}
It may be guessed that \ref{item:m5} is equivalent, e.g., also to the following:
\begin{enumerate}[label={\rm (\textsc{m}\arabic{*})}]
 \setcounter{enumi}{10}
\item \label{item:m1prime} 
There exists a sequence $x \in \mathscr{B}$ such that both  \ref{item:c2} and \ref{item:c5} hold\textup{.}
\end{enumerate}
However we can prove that this is false by the following example. 
Let $\mathcal{I}$ be a maximal ideal. 
Since $\mathcal{I}$ is maximal, there exists a unique $A \in \{2\mathbf{N}+1,2\mathbf{N}+2\}$ such that $A\in \mathcal{I}$. 
Accordingly, let $x$ be the sequence defined by $x_n=n$ if $n \in A$ and $x_n=0$ otherwise. 
Then $x \to_{\mathcal{I}} 0$ and $\Gamma_x(\mathcal{I})=\mathrm{L}_x=\{0\}$, hence $x \in \mathscr{B}$ and \ref{item:c5} holds. 
In addition, 
$$
\Sigma_x(\mathcal{I})=\{\sigma \in \Sigma: \{n \in \mathbf{N}: x_{\sigma(n)}\ge 1\} \in \mathcal{I}\}=\{\sigma \in \Sigma: \sigma^{-1}[A]\in \mathcal{I}\}.
$$
It follows by \cite[Example 2.6]{MR4165727} that $\Sigma_x(\mathcal{I})$ is not meager, hence also \ref{item:c2} holds. 
(Note that the same argument shows that $\Sigma_x(\mathcal{I})$ is not comeager, hence \ref{item:c1} fails.) 
This proves that \ref{item:m1prime} is verified, while \ref{item:m5} does not. (The same example works with $\mathcal{I}$-limit points.)
\end{rmk}


\section{Proof of Theorem \ref{thm:theoremlimit}} \label{sec:proofs}

The following result strenghtens \cite[Lemma 3.3]{MR4165727}.
\begin{lem}\label{lem:keylemma}
Let $x$ be a sequence taking values in a first countable space $X$ and let $\mathcal{I}$ be a meager ideal on $\mathbf{N}$. Then 
$$
\{\sigma \in \Sigma: \eta \in \Lambda_{\sigma(x)}(\mathcal{I})\}
\,\,\,\,\text{ and }\,\,\,\,
\{\pi \in \Pi: \eta \in \Lambda_{\pi(x)}(\mathcal{I})\}
$$
are comeager for each $\eta \in \mathrm{L}_x$. 
\end{lem}
\begin{proof}
Assume that $\mathrm{L}_x\neq \emptyset$, otherwise there is nothing to prove. Fix $\eta \in \mathrm{L}_x$ and let $(U_{\eta,m})$ be a decreasing local base at $\eta$. Since $\mathcal{I}$ is a meager ideal, it follows by \cite[Theorem 2.1]{MR579439} that there exists an increasing sequence $(\iota_n)$ of positive integers such that $A\notin \mathcal{I}$ whenever $I_n:=[\iota_n,\iota_{n+1}) \subseteq A$ for infinitely many $n$. For each $m \in \mathbf{N}$ define
\begin{equation}\label{eq:Smeta}
S_m(\eta):=\bigcup\nolimits_{k\ge m}\{\sigma \in \Sigma: 
x_{\sigma(n)} \in U_{\eta,m} \text{ for all }n \in I_k\}.
\end{equation}
Since $\bigcap\nolimits_m S_m(\eta)$ is contained in $\{\sigma \in \Sigma: \eta \in \Lambda_{\sigma(x)}(\mathcal{I})\}$ by the previous observation, it will be sufficient to show that each $S_m(\eta)$ is comeager. To this aim, fix $m \in \mathbf{N}$ and let $C=\{\sigma \in \Sigma: \sigma(1)=a_1,\ldots,\sigma(k)=a_k\}$ be a basic open set, for some positive integers $a_1<\cdots<a_k$. Note that the set $E:=\{n \in \mathbf{N}: x_{n}\in U_{\eta,m}\}$ is infinite since $\eta \in \mathrm{L}_x$. At this point, it is enough to see that 
$$
\{\sigma \in C: \sigma(n)=e_n \text{ for all }n\text{ with }k< n<\iota_{a_k+m}\},
$$
for some increasing values $e_n \in E\setminus [1,a_k]$, 
is a nonempty open set contained in $C \cap S_m(\eta)$. This proves that $S_m(\eta)$ is comeager, completing the first part. 
The second part of the proof about permutations goes verbatim. 
\end{proof}

We are ready for the proof of Theorem \ref{thm:theoremlimit}.
\begin{proof}[Proof of Theorem \ref{thm:theoremlimit}] 
\ref{item:l1} $\implies$ \ref{item:l2} It is obvious.

\ref{item:l2} $\implies$ \ref{item:l5} Suppose that there exists $\ell \in \mathrm{L}_x \setminus \Lambda_x(\mathcal{I})$. Then $\tilde{\Sigma}_x(\mathcal{I})$ is contained in $\Sigma\setminus \{\sigma \in \Sigma: \eta \in \Lambda_{\sigma(x)}(\mathcal{I})\}$, which is meager by Lemma \ref{lem:keylemma}.

\ref{item:l5} $\implies$ \ref{item:l1} 
Suppose that $\mathrm{L}_x \neq \emptyset$, otherwise the claim is trivial. 
Let $\mathscr{L}$ be a countable dense subset of $\mathrm{L}_x$, and denote by $(U_{\eta,m}: m\ge 1)$ a decreasing local base at each $\eta \in X$. 
As it is shown in the proof of Lemma \ref{lem:keylemma}, the set $S_{m,\eta}$ defined in \eqref{eq:Smeta} is comeager for all $m \in \mathbf{N}$ and $\eta \in \mathrm{L}_x$, therefore also 
$
S:=\bigcap\nolimits_{m \ge 1}\bigcap\nolimits_{\eta \in \mathscr{L}}S_{m}(\eta)
$ is comeager. Note that, equivalently, 
\begin{equation}\label{eq:defS}
S=
\{\sigma \in \Sigma: \forall \eta \in \mathscr{L}, \exists^\infty k \in \mathbf{N}, \forall n \in I_k, x_{\sigma(n)} \in U_{\eta,m}\}.
\end{equation}
Since $S\subseteq \{\sigma \in \Sigma: \eta \in \Lambda_{\sigma(x)}(\mathcal{I}) \text{ for all }\eta \in \mathscr{L}\}$, we have that $\mathscr{L}\subseteq \Lambda_{\sigma(x)}(\mathcal{I})$ for all $\sigma \in S$. 
To conclude the proof, we claim that $S$ is contained in $\tilde{\Sigma}_x(\mathcal{I})$. 

To this aim, fix $\sigma \in S$. On the one hand, we have 
$$
\Lambda_{\sigma(x)}(\mathcal{I}) \subseteq \mathrm{L}_{\sigma(x)}\subseteq \mathrm{L}_x=\Lambda_x(\mathcal{I}).
$$
Conversely, fix $\eta \in \mathrm{L}_x$. Since $\mathscr{L}$ is dense, there exists a sequence $(\eta_t)$ in $\mathscr{L}$ which is convergent to $\eta$. Without loss of generality we can assume that $\eta_t \in U_{\eta,t}$ for all $t \in \mathbf{N}$. 
At this point, for each $t\in \mathbf{N}$, there exists $m_t \in \mathbf{N}$ such that $U_{\eta_t,m_t}\subseteq  U_{\eta,t}$. By the explicit expression \eqref{eq:defS}, $S$ is contained in 
$$
\{\sigma \in \Sigma: \forall t \in \mathbf{N}, \exists^\infty k \in \mathbf{N}, \forall n \in I_k, x_{\sigma(n)} \in U_{\eta_t,m_t}\}
$$ 
It follows that the subsequence $\sigma(x)$ has a subsubsequence $\hat{\sigma}(\sigma(x))$ which is convergent to $\eta$ and such that $\hat{\sigma}[\mathbf{N}]$ contains infinitely many intervals $I_k$, hence $\eta$ is an $\mathcal{I}$-limit point of $\sigma(x)$. Therefore  
$$
\Lambda_x(\mathcal{I})=\mathrm{L}_x \subseteq \Lambda_{\sigma(x)}(\mathcal{I}).
$$ 
This proves that $S\subseteq \tilde{\Sigma}_x(\mathcal{I})$, completing the proof. 

The proof of \ref{item:l3} $\implies$ \ref{item:l4} $\implies$ \ref{item:l5} $\implies$ \ref{item:l3} goes verbatim.
\end{proof}

\section{Proof of Theorem \ref{cor:characterizationmeager}}

The following result will be the key tool in the proof of the characterization of meager ideals.
\begin{thm}\label{thm:newszymon}
Let $\mathcal{I}$ be an ideal on $\mathbf{N}$ and $x$ be a nonconvergent sequence in a topological space $X$ 
such that 
there exists $\eta \in X$ for which 
\begin{equation}\label{eq:conditionaccuulation}
\{\sigma \in \Sigma: \eta \in \Gamma_{\sigma(x)}(\mathcal{I})\} 
\,\,\,\text{ or }\,\,\,
\{\pi \in \Pi: \eta \in \Gamma_{\pi(x)}(\mathcal{I})\}
\end{equation}
is comeager. 
Then $\mathcal{I}$ is meager. 
\end{thm}
\begin{proof}
First, assume that we can fix $\eta \in X$ such that $S:=\{\sigma \in \Sigma: \eta \in \Gamma_{\sigma(x)}(\mathcal{I})\}$ is comeager. Hence there exists a decreasing sequence $(G_n)$ of dense open subsets of $\Sigma$ such that $\bigcap\nolimits_n G_n \subseteq S$. In addition, since $x$ is not convergent to $\eta$, there is a neighborhood $U$ of $\eta$ such that $E:=\{n \in \mathbf{N}: x_n \notin U\}$ is infinite. 

At this point, consider the following game 
defined by Laflamme in \cite{MR1367134}: Players I and II choose alternatively subsets $C_1,F_1,C_2,F_2,\ldots$ of $\mathbf{N}$, where the sets $C_1\supseteq C_2\supseteq \ldots$, which are chosen by Player I, are cofinite and the sets $F_k\subseteq C_k$, which are chosen by Player II, are finite. 
Player II is declared to be the winner if $\bigcup\nolimits_k F_k \notin \mathcal{I}$. 
Note that we may suppose without loss of generality that 
$F_k\cap C_{k+1}=\emptyset$ and 
$C_k=[c_k,\infty)$ 
for all $k \in \mathbf{N}$. 
Thanks to \cite[Theorem 2.12]{MR1367134}, Player II has a winning strategy if and only if $\mathcal{I}$ is meager. 
Hence, the rest of the proof consists in showing that Player II has a winning strategy. 

To this aim, we will define recursively, together with the description of the strategy of Player II, also a decreasing sequence of basic open sets $A_1\supseteq B_1 \supseteq A_2\supseteq B_2 \supseteq \cdots$ in $\Sigma$ (recall that a basic open set in $\Sigma$ is a cylinder of the type $D=\{\sigma \in \Sigma: \sigma(1)=a_1,\ldots,\sigma(n)=a_n\}$ for some positive integers $a_1<\cdots<a_n$, and we set $m(D):=a_n$). 
Suppose that the sets $C_1,F_1,\ldots,C_{k-1},F_{k-1},C_k\subseteq \mathbf{N}$ have been already chosen and that the open sets $A_1, B_1, \ldots, A_{k-1}, B_{k-1} \subseteq \Sigma$ have already been defined, for some $k \in \mathbf{N}$, where we assume by convention that $B_0:=\Sigma$ and $m(\Sigma):=0$. 
Then we define the sets $A_k, B_k$, and $F_k$ as it follows: 
\begin{enumerate}[label={\rm (\roman{*})}]
\item $A_k:=\{\sigma \in B_{k-1}: \sigma(n)=e_n \text{ for all }n\text{ with }m(B_{k-1})<n<c_k\}$ for some increasing values $e_n \in E$ which are bigger than $\sigma(m(B_{k-1}))$ (note that this is possible since $E$ is infinite);
\item $B_k$ is a nonempty basic open set contained in $G_k \cap A_k$ (note that this is possible since $G_k$ is open dense and $A_k$ is nonempty open);
\item $F_k:=\{n \in \mathbf{N}: x_{\sigma(n)} \in U\} \cap [c_k, m(B_k)]$ (note that this is a finite set and it is possibly empty). 
\end{enumerate} 
We obtain by construction that there exists $\sigma^\star \in \Sigma$ such that 
$$
\sigma^\star \in \bigcap\nolimits_k B_k \subseteq \bigcap\nolimits_k G_k\subseteq S,
$$
so that $\eta$ is an $\mathcal{I}$-cluster point of the subsequence $\sigma^\star(x)$, which implies that $\{n \in \mathbf{N}: x_{\sigma^\star(n)} \in U\} \notin \mathcal{I}$. 
At the same time, by the definitions above we get 
$$
\{n \in \mathbf{N}: x_{\sigma^\star(n)} \in U\} 
=\bigcup\nolimits_k \{n \in [c_k, m(B_k)]: x_{\sigma^\star(n)} \in U\}
=\bigcup\nolimits_k F_k.
$$
This proves that Player II has a winning strategy, concluding the first part of the proof. 

For the second part, 
recall that a basic open set in $\Pi$ is a cylinder of the type $D=\{\pi \in \Pi: \pi(1)=a_1,\ldots,\pi(n)=a_n\}$ for some distinct $a_1,\ldots,a_n \in \mathbf{N}$, and set $m(D):=\max\{a_1,\ldots,a_n\}$. 
Minor modifications are necessary in the definitions of the corresponding sets $A_k$ and $B_k$ as it follows: 
\begin{enumerate}[label={\rm (\roman{*}$^\prime$)}]
\item $A_k:=\{\pi \in B_{k-1}: \pi(n)=e_n \text{ for all }n\text{ with }m(B_{k-1})<n<c_k\}$ for the smallest possible values of $e_n \in E$ which have not been chosen before in the previous steps; 
\item $B_k$ is a nonempty basic open set contained in $G_k \cap A_k$ with the additional condition that if $\pi \in B_k$ then $\{\pi(1),\ldots,\pi(m(B_k))\}$ coincides with $\{1,\ldots,m(B_k)\}$ (this is still possible replacing $B_k$, if necessary, with a smaller subset which verifies this condition). 
\end{enumerate} 
It follows by construction that $\bigcap\nolimits_k B_k$ contains an element $\pi^\star$ which is really a permutation on $\mathbf{N}$. Finally, the proof of the permutations case goes along the same lines as above. 
\end{proof}

\begin{cor}\label{cor:newszymonlimit}
Let $\mathcal{I}$ be an ideal on $\mathbf{N}$ and $x$ be a nonconvergent sequence in a topological space $X$  
such that 
there exists $\eta \in X$ for which 
$$
\{\sigma \in \Sigma: \eta \in \Lambda_{\sigma(x)}(\mathcal{I})\} 
\,\,\,\text{ or }\,\,\,
\{\pi \in \Pi: \eta \in \Lambda_{\pi(x)}(\mathcal{I})\}
$$
is comeager. Then $\mathcal{I}$ is meager. 
\end{cor}
\begin{proof}
This follows by Theorem \ref{thm:newszymon} and the fact that every $\mathcal{I}$-limit point is necessarily an $\mathcal{I}$-cluster point. 
\end{proof}

Now, we show that the converse of Theorem \ref{thm:poplawski} holds, provided that the sequence $x$ has a dense image. 
\begin{thm}\label{thm:converseseries}
Let $x \in \mathscr{C}$ such that $\Sigma_{S,x}(\mathcal{I})$ is meager. Then $\mathcal{I}$ is meager.
\end{thm}
\begin{proof}
We follow the same lines of the proof of Theorem \ref{thm:newszymon}. 
Suppose that $\Sigma\setminus \Sigma_{S,x}(\mathcal{I})$ is comeager, hence it contains $\bigcap_n G_n$, where $(G_n)$ is a decreasing sequence of dense open sets in $\Sigma$. 
Using the same Laflamme game, we modify the construction of the sets $A_k$ and $F_k$ as it follows: 
\begin{enumerate}[label={\rm (\roman{*}$^{\prime\prime}$)}]
\item $A_k:=\{\sigma \in B_{k-1}: \sigma(n)=e_n \text{ for all }n\text{ with }m(B_{k-1})<n<c_k\}$, where the increasing values $e_n$ are chosen such that $|S_n\sigma(x)|<1$ 
(note that this is possible since $x$ has a dense image);
\end{enumerate}
\begin{enumerate}[label={\rm (\roman{*}$^{\prime\prime}$)}]
\setcounter{enumi}{2}
\item $F_k=:\{n \in \mathbf{N}: |S_n\sigma(x)|\ge 1\} \cap [c_k, m(B_k)]$.
\end{enumerate}
Similarly, we obtain that there exists $\sigma^\star \in \bigcap_k B_k \subseteq \Sigma\setminus \Sigma_{S,x}(\mathcal{I})$, so that $S\sigma^\star(x)$ is not $\mathcal{I}$-bounded. It follows that 
\begin{displaymath}
\begin{split}
\bigcup\nolimits_k F_k
=\bigcup\nolimits_k \{n \in [c_k, m(B_k)]: |S_n\sigma^\star(x)|\ge 1\}=
\{n \in \mathbf{N}: |S_n\sigma^\star(x)|\ge 1\}\notin \mathcal{I},
\end{split}
\end{displaymath}
which proves that player II has a winning strategy.
\end{proof}

We can finally proceed to the proof of Theorem \ref{cor:characterizationmeager}.
\begin{proof}[Proof of Theorem \ref{cor:characterizationmeager}] 
\ref{item:m1} $\implies$ \ref{item:m2} 
First, note that the ideal $\mathcal{I}$ cannot be maximal: indeed, in the opposite, thanks to Remark \ref{rmk:nostrongercorollary}, there exists a real sequence $x$ such that \ref{item:c5} holds and \ref{item:c1} does not. Since $\mathcal{I}$ has to be not maximal, there exists a set $A\notin \mathcal{I}$ such that $A^c \not \in \mathcal{I}$. At this point, it follows that condition \ref{item:c5} holds for the real sequence $x$ defined by $x_n=1$ if $n \in A$ and $x_n=0$ otherwise. Hence \ref{item:c1} holds too by \ref{item:m1}, proving the implication.

\smallskip

\ref{item:m2} $\implies$ \ref{item:m5} Assume that $\Gamma_{x}(\mathcal{I})=\mathrm{L}_x \neq \emptyset$. Then 
$$
\Sigma_x(\mathcal{I})
=\{\sigma \in \Sigma: \Gamma_{\sigma(x)}(\mathcal{I})=\mathrm{L}_x\}
=\bigcap\nolimits_{\eta \in \mathrm{L}_x}\{\sigma \in \Sigma: \eta \in \Gamma_{\sigma(x)}(\mathcal{I})\}.
$$
Hence \eqref{eq:conditionaccuulation} holds and the implication follows by Theorem \ref{thm:newszymon}.

\smallskip

\ref{item:m5} $\implies$ \ref{item:m1} It follows by Theorem \ref{thm:theoremcluster}. 

\medskip

The proof of \ref{item:m3} $\implies$ \ref{item:m4} $\implies$ \ref{item:m5} $\implies$ \ref{item:m3} goes verbatim as above.

\medskip

The proofs of 
\ref{item:m1L} $\implies$ \ref{item:m2L} $\implies$ \ref{item:m5} $\implies$ \ref{item:m1L} 
and 
\ref{item:m3L} $\implies$ \ref{item:m4L} $\implies$ \ref{item:m5} $\implies$ \ref{item:m3L} 
go along the same lines, replacing Theorem \ref{thm:theoremcluster} and Theorem \ref{thm:newszymon} with Theorem \ref{thm:theoremlimit} and Corollary \ref{cor:newszymonlimit}, respectively. 

\medskip

The equivalence \ref{item:mSeries} $\Longleftrightarrow$ \ref{item:m5} follows by Theorem \ref{thm:poplawski} and Theorem \ref{thm:converseseries}.
\end{proof}


\section{Concluding Remarks}

Let us identify each $\sigma \in \Sigma$ with the real number $\sum_{n\ge 1} 2^{-\sigma(n)}$. This provides us a homeomorphism $h:\Sigma\to (0,1]$. Denote by $\lambda$ the Lebesgue measure on $(0,1]$ and let
$\hat{\lambda}:=\lambda \circ h$ be its pushforward on $\Sigma$. 

At this point, one may hope in a measure analogue of Theorem \ref{cor:characterizationmeager}. However, for the subsequences case, this does not seem to be possible as it follows from the next two examples: 
(i) the ideal $\mathcal{I}:=\{A\subseteq \mathbf{N}: \sum_{a \in A}1/a<\infty\}$ is $F_\sigma$ and, thanks to \cite[Proposition 2.3 and Theorem 3.1]{MR3879311}, we have $\hat{\lambda}(\Sigma_x(\mathcal{I}))=1$ for all real sequences $x$; (ii) the Fubini sum $\mathcal{P}(\mathbf{N})\oplus \mathrm{Fin}$, which can be identified with $\mathcal{J}:=\{A\subseteq \mathbf{N}: |A\cap (2\mathbf{N}+1)|<\infty\}$, is also $F_\sigma$ and, thanks to \cite[Example 2]{MR3568092}, there exists a real sequence $x$ such that $\hat{\lambda}(\Sigma_x(\mathcal{J}))=0$. 

On the other hand, a [non]analogue for the permutations is more difficult to obtain due to the lack of a natural candidate for a measure on $\Pi$. 
Indeed, since $\Pi$ is a Polish group which is not locally compact, we cannot speak about Haar measure. 
An alternative could be to consider the notion of \emph{prevalent set} as introduced by Christensen in \cite{MR326293}: a set $S\subseteq \Pi$ is called prevalent if it is universally measurable (i.e., measurable with respect to every complete probability measure on $\Pi$ that measures all Borel subsets of $\Pi$) and there exists a (not necessarily invariant) Borel probability measure $\mu$ over $\Pi$ such that $\mu(\zeta S)=1$ for all permutations $\zeta \in \Pi$, cf. also \cite{MR1271480}. 
We leave as open question for the interested reader to check whether $\Pi_x(\mathcal{I})$ and/or $\tilde{\Pi}_x(\mathcal{I})$ are prevalent sets whenever $\mathcal{I}$ is a meager ideal.

\bigskip

\bibliographystyle{amsplain}
\bibliography{subsequences}

\providecommand{\MR}[1]{}
\providecommand{\bysame}{\leavevmode\hbox to3em{\hrulefill}\thinspace}
\providecommand{\MR}{\relax\ifhmode\unskip\space\fi MR }
\providecommand{\MRhref}[2]{%
  \href{http://www.ams.org/mathscinet-getitem?mr=#1}{#2}
}
\providecommand{\href}[2]{#2}
\begin{thebibliography}{10}

\bibitem{MR3568092}
M.~Balcerzak, Sz. G{\l}\c{a}b, and A.~Wachowicz, \emph{Qualitative properties
  of ideal convergent subsequences and rearrangements}, Acta Math. Hungar.
  \textbf{150} (2016), no.~2, 312--323. \MR{3568092}

\bibitem{MR3883171}
M.~Balcerzak and P.~Leonetti, \emph{On the relationship between ideal cluster
  points and ideal limit points}, Topology Appl. \textbf{252} (2019), 178--190.
  \MR{3883171}

\bibitem{MR4165727}
\bysame, \emph{The {B}aire category of subsequences and permutations which
  preserve limit points}, Results Math. \textbf{75} (2020), no.~4, Paper No.
  171, 14. \MR{4165727}

\bibitem{MR3712964}
M.~Balcerzak, M.~Pop{\l}awski, and A.~Wachowicz, \emph{The {B}aire category of
  ideal convergent subseries and rearrangements}, Topology Appl. \textbf{231}
  (2017), 219--230. \MR{3712964}

\bibitem{MR3989357}
\bysame, \emph{Ideal convergent subseries in {B}anach spaces}, Quaest. Math.
  \textbf{42} (2019), no.~6, 765--779. \MR{3989357}

\bibitem{MR2735533}
A.~Bartoszewicz, Sz. G{\l}{\c{a}}b, and A.~Wachowicz, \emph{Remarks on ideal
  boundedness, convergence and variation of sequences}, J. Math. Anal. Appl.
  \textbf{375} (2011), no.~2, 431--435. \MR{2735533}

\bibitem{MR9997}
R.~C. Buck, \emph{Limit points of subsequences}, Bull. Amer. Math. Soc.
  \textbf{50} (1944), 395--397. \MR{9997}

\bibitem{MR326293}
J.~P.~R. Christensen, \emph{On sets of {H}aar measure zero in abelian {P}olish
  groups}, Israel J. Math. \textbf{13} (1972), 255--260 (1973). \MR{326293}

\bibitem{MR1271480}
R.~Dougherty and J.~Mycielski, \emph{The prevalence of permutations with
  infinite cycles}, Fund. Math. \textbf{144} (1994), no.~1, 89--94.
  \MR{1271480}

\bibitem{MR2135846}
A.~Faisant, G.~Grekos, and V.~Toma, \emph{On the statistical variation of
  sequences}, J. Math. Anal. Appl. \textbf{306} (2005), no.~2, 432--439.
  \MR{2135846}

\bibitem{MR1367134}
C.~Laflamme, \emph{Filter games and combinatorial properties of strategies},
  Set theory ({B}oise, {ID}, 1992--1994), Contemp. Math., vol. 192, Amer. Math.
  Soc., Providence, RI, 1996, pp.~51--67. \MR{1367134}

\bibitem{MR3879311}
P.~Leonetti, \emph{Thinnable ideals and invariance of cluster points}, Rocky
  Mountain J. Math. \textbf{48} (2018), no.~6, 1951--1961. \MR{3879311}

\bibitem{Leo17b}
\bysame, \emph{Invariance of ideal limit points}, Topology Appl. \textbf{252}
  (2019), 169--177.

\bibitem{MR3968131}
\bysame, \emph{Limit points of subsequences}, Topology Appl. \textbf{263}
  (2019), 221--229. \MR{3968131}

\bibitem{MR3950736}
P.~Leonetti, H.~I. Miller, and L.~Miller-Van~Wieren, \emph{Duality between
  measure and category of almost all subsequences of a given sequence}, Period.
  Math. Hungar. \textbf{78} (2019), no.~2, 152--156. \MR{3950736}

\bibitem{MR1260176}
H.~I. Miller, \emph{A measure theoretical subsequence characterization of
  statistical convergence}, Trans. Amer. Math. Soc. \textbf{347} (1995), no.~5,
  1811--1819. \MR{1260176}

\bibitem{MR4089884}
H.~I. Miller and L.~Miller-Van~Wieren, \emph{Statistical cluster point and
  statistical limit point sets of subsequences of a given sequence}, Hacet. J.
  Math. Stat. \textbf{49} (2020), no.~2, 494--497. \MR{4089884}

\bibitem{MR1924673}
H.~I. Miller and C.~Orhan, \emph{On almost convergent and statistically
  convergent subsequences}, Acta Math. Hungar. \textbf{93} (2001), no.~1-2,
  135--151. \MR{1924673}

\bibitem{MR4183385}
L.~Miller-Van~Wieren, E.~Ta\c{s}, and T.~Yurdakadim, \emph{Category theoretical
  view of {$I$}-cluster and {$I$}-limit points of subsequences}, Acta Comment.
  Univ. Tartu. Math. \textbf{24} (2020), no.~1, 103--108. \MR{4183385}

\bibitem{MR579439}
M.~Talagrand, \emph{Compacts de fonctions mesurables et filtres non
  mesurables}, Studia Math. \textbf{67} (1980), no.~1, 13--43. \MR{579439}

\end{thebibliography}
%
%

\end{document}